\documentclass[12pt]{amsart}

\usepackage{amsmath} 
\usepackage{amssymb} 
\usepackage{amsthm} 
\usepackage{dsfont} 
\usepackage{graphicx} 

\newtheorem{theorem}{Theorem}[section]

\newtheorem{definition}[theorem]{Definition}
\newtheorem{example}[theorem]{Example}
\newtheorem{lemma}[theorem]{Lemma}

\begin{document}

\title[Sudoku Duality]{Duality for Sudoku}

\author{Thomas Fischer}
\address{Fuchstanzstr. 20, 60489 Frankfurt am Main, Germany.}

\email{dr.thomas.fischer@gmx.de}

\date{}     


\begin{abstract}
We consider a mathematical model for the classical Sudoku puzzle, which we call the primal 
problem and introduce a corresponding dual problem. Both problems are constraint satisfaction
models and a duality relation between them is proved. Based on these models, we 
introduce a primal and a dual optimization problem and show weak and strong duality properties.
\end{abstract}




\maketitle

                                        %
                                        %
\section{Introduction} \label{S:intro} 
A Sudoku is a square consisting of a 9$\times$9 grid which is
partly pre-populated by numbers between 1 and 9 called the givens.
The problem consists of finding numbers between 1 and 9 for all
unpopulated cells, such that each row, each column and each block 
consists of exactly the numbers $1, \ldots, 9$. The blocks of a Sudoku 
partition the Sudoku square into subsquares of size 3$\times$3. Each 
Sudoku consists of 9 rows, 9 columns and 9 blocks.

In \cite{Fis} we introduced a mathematical model for this Sudoku puzzle and called it the 
generalized Sudoku problem. As we are concerned throughout this paper with duality, we call it 
in Section \ref{S:PP} the primal problem. We introduce a dual problem in Section \ref{S:DP}
and show the relation between the primal and the dual. This relation will be established using
a necessary solution condition developed in \cite{Fis} and can be interpreted as a duality result. 
But the primal and the dual problem defined in this way do not allow to describe 
duality results considering duality gaps. Therefore we introduce in Section \ref {S:opt} primal and dual 
optimization problems and show how they replace the original problems. In section \ref{S:dual}
we prove a weak and a strong duality property between the primal and the dual optimization problem.

The primal and the dual problem are of a type, which is often called constraint satisfaction problem 
or CSP. For a description of general CSPs with examples, solution techniques and applications see 
the survey article of Dechter and Rossi \cite{DR}. Duality statements are standard properties of linear 
and nonlinear programs. An overview with several examples and applications can be found in the 
book of Boyd and Vandenberghe \cite{BV}. The linear case has been treated by Dantzig and Thapa
\cite{DT}.

Finally, we collect some basic terms and notations. Let $\mathds{Z}$ denote the set of integers.
Let the $n$-times cartesian product of any set be indicated by a superscript $n$, i.e., $\mathds{Z}^n$ 
denote the $n$-times cartesian product of $\mathds{Z}$. The vectors $\mathbf{0}$ 
respectively $\mathbf{1}$ denote the zero respectively one vector, consisting of zeros respectively 
ones in each component. The number of components of these vectors is often indicated by an index. 
Each vector is considered to be a column vector. $U$ denotes the identity matrix. The transpose 
of a vector or a matrix is indicated by a superscript $T$. The sign function is denoted by $sgn$. 
We consider the sum over an empty index set to be zero. The brackets with index $[]_{i}$ denote 
the $i^{th}$ component of a vector contained in the brackets.
The symbol $\sharp$ denotes the number of elements (cardinality) of a finite set.

                                        %
                                        %
\section{The Primal Problem} \label{S:PP}
We replicate here the definition of the generalized Sudoku problem as introduced in \cite{Fis} 
and call it this time the primal problem. Let $n$ be an integer with $n \ge 1$. We define the sum
\[                                                   
s(n)  = \sum_{i=1}^{n-1}i
\]
and define a matrix $A(n)$ with $s(n)$ rows and $n$ columns inductively.
For $n=1$, let $A(1)$ denote the empty matrix, i.e., a matrix without entries.
Assume the matrix $A(n-1)$ had been defined with $s(n-1)$ rows 
and $n-1$ columns. Then we set

{
\renewcommand{\arraystretch}{2.0}     
\[
A(n) = 
\begin{pmatrix}    
\begin{array}{c|c}
\mathbf{1}_{n-1} & - U_{n-1} \\
\hline
\mathbf{0}_{s(n-1)} & A(n-1)
\end{array}
\end{pmatrix}.
\]
}

We extend the matrix $A(n)$ to a matrix $A$ with $n \cdot s(n)$ rows and 
$n^2$ columns. The matrix $A$ consists in the ``main diagonal" of $n$ 
matrices $A(n)$ and the remaining values are set to zero. The matrix $A$ 
depends on the value $n$, but we do not state this dependence explicitly. 

Given the set $\{1, \ldots, n^2\} \subset \mathds{Z}$, let $\pi$ be any
permutation on this set, i.e., 
\[
\pi: \{1, \ldots, n^2\} \longrightarrow \{1, \ldots, n^2\}
\]
be a permutation. We extend the notion of permutation to the matrix $A$,  
i.e., we define $\pi (A) = (a^{\pi^{-1}(1)}, \ldots, a^{\pi^{-1}(n^2)})$, 
where $a^j$ denotes the $j^{th}$ column of $A$ for $j = 1, \ldots , n^2$. 
Given a permutation $\pi$ on $\{1, \ldots, n^2\}$, we define the matrix 
$A_\pi = \pi (A)$, i.e., we interchange the columns of $A$ according to 
the permutation $\pi$. 

\begin{definition} \label{D:nonzero}
Let $s \ge 1$. For any point $y = (y_1, \ldots, y_s)^T \in \mathds{Z}^s$
we write $y < > \mathbf{0}$ if each component of $y$ is nonzero, 
i.e., if $y_i \ne 0$ for $i=1, \ldots, s$.
\end{definition}

This definition should not be confused with the expression $y \ne \mathbf{0}$, 
where only one component of $y$ has to be nonzero.

Given is $n \ge 2$, some permutations $\pi_1$, $\pi_2$, $\pi_3$ on $\{1, \ldots, n^2\}$, 
some $0 \le k \le n^2$, an index set $\{i_1, \ldots, i_k \} \subset \{ 1, \ldots, n^2\}$, 
and givens $g_{i_1}, \ldots , g_{i_k} \in\mathds{Z}$ with $1 \le g_{i_l} \le n$ for 
$l = 1, \ldots, k$.

Let $A_{eq}$ be the $k \times n^2$ matrix, which consists of the rows $i_1, \ldots, i_k$
of the identity matrix $U_{n^2}$, i.e., the $i_l^{th}$ component of the $l^{th}$ row of $A_{eq}$ is 
equal to $1$ (and zero otherwise). In other words $A_{eq}$ is defined by $[A_{eq}x]_l = x_{i_l}$
for each $x = (x_1, \ldots, x_{n^2})^T \in \mathds{Z}^{n^2}$ and $l=1, \ldots, k$.
The vector of ones $\mathbf{1}_{n^2}$ is mapped to the vector of ones $\mathbf{1}_{k}$
by $A_{eq}$, i.e., $A_{eq}\mathbf{1}_{n^2} = \mathbf{1}_{k}$. 
Define $g=(g_{i_1}, \ldots , g_{i_k})^T \in\mathds{Z}^k$.

Now, we are in position to state the primal problem.
\begin{align*}
(PP) \qquad
& \mbox{Find }x = (x_1, \ldots, x_{n^2})^T \in \mathds{Z}^{n^2} \mbox{ such that } \\
& 1 \le x_i \le n \mbox{ for }i = 1, \ldots , n^2, \\
& A_{\pi_r} x <> \mathbf{0} \mbox{ for }r=1, 2, 3 \mbox{ and} \\
& A_{eq}x = g. 
\end{align*}

We restrict ourselves to this mathematical model and do not refer directly to the classical Sudoku 
puzzle. In particular, we will not investigate the relation of this model to the Sudoku puzzle
in detail. This had been described in \cite{Fis} already.

Another problem modeling Sudoku had been introduced by Kaibel and Koch \cite{KK}. Their linear model 
consisted of 0-1-variables and contained equality constraints. The same type of problem had been considered 
by Provan \cite{Pro}. Both did not consider duality properties.

                                        %
                                        %
\section{The Dual Problem} \label{S:DP}
We introduce the dual problem.
\begin{align*}
(DP) \qquad
& \mbox{Find }\lambda \in \{-1, +1\}^{n \cdot s(n)} \mbox{ such that } \\
& A_{\pi_r}A^T_{\pi_1}\lambda <> \mathbf{0} \mbox{ for }r = 1, 2, 3 \mbox{ and } \\
& A_{eq}A^T_{\pi_1}\lambda = 2 g - (n+1) \mathbf{1}_k. 
\end{align*}

The last condition $A_{eq}A^T_{\pi_1}\lambda = 2 g - (n+1) \mathbf{1}_k$ can be 
reformulated using a componentwise description
\[
[A^T_{\pi_1}\lambda]_{i_l} = 2 g_{i_l} - (n+1) \mbox{ for }l=1, \ldots, k.
\]

The dual problem is closely related to the generalized sign function introduced in \cite{Fis}.
The generalized sign function is based on the classical sign function and is also denoted by $sgn$.

\begin{definition}     \label{D:sgn}
Let $s \ge 1$. For any point $y = (y_1, \ldots, y_s)^T \in \mathds{Z}^s$ 
with $y<> \mathbf{0}$ we define the generalized sign function 
$sgn: \mathds{Z}^s \longrightarrow \mathds{Z}^s $ by
\[
sgn ( y ) =
\begin{pmatrix}
sgn ( y_1 ) \\
\vdots \\
sgn ( y_s )
\end{pmatrix}.
\]
\end{definition}

We continue with some preparing lemmas.

\begin{lemma} \label{L:31}
Let $\pi$ be a permutation on $\{1,  \ldots, n^2\}$ and let 
$x = (x_1, \ldots, x_{n^2})^T \in \mathds{Z}^{n^2}$, such that $1 \le x_i \le n$ for 
$i=1, \ldots, n^2$ and $A_\pi x <> \mathbf{0}$. Then $\lambda = sgn( A_{\pi_1} x)$
satisfies $\lambda \in \{-1, +1\}^{n \cdot s(n)}$ and $A_\pi A^T_{\pi_1}\lambda <> \mathbf{0}$. 
\end{lemma}
\begin{proof} 
The property $\lambda \in \{-1, +1\}^{n \cdot s(n)}$ follows from $A_\pi x <> \mathbf{0}$ and 
the definition of $sgn$. Using \cite[Theorem 5.1]{Fis} and the equation
$A_{\pi}\mathbf{1}_{n^2} = \mathbf{0}$ we obtain
\[
\begin{array}{lll}
A_\pi A_{\pi_1}^T \lambda 
& = &A_\pi A_{\pi_1}^T sgn( A_{\pi_1} x) \\
& = &A_\pi (A_{\pi_1}^T sgn( A_{\pi_1} x) + (n+1) \mathbf{1}_{n^2}) 
- (n+1) A_\pi \mathbf{1}_{n^2} \\
& = &2 A_\pi x \\
& <> &\mathbf{0},
\end{array}
\]
which is the desired result.
\end{proof} 

\begin{lemma} \label{L:32}
Let $\pi$ be a permutation on $\{1, \ldots, n^2\}$ and let $\lambda$ be a point in 
$\{-1, +1\}^{n \cdot s(n)}$. Then $-(n-1) \le [A_\pi^T \lambda]_i \le n-1$ 
for $i=1, \ldots, n^2$.
\end{lemma}
\begin{proof} 
We divide the proof of this lemma into three steps. First we prove it for the matrix $A(n)$, 
then for $A$ and, finally, for $A_\pi$. The first claim reads as  
$-(n-1) \le [A(n)^T \lambda]_i \le n-1$ for $\lambda \in \{-1, +1\}^{s(n)}$ and 
$i=1, \ldots, n$. We prove this claim by induction on $n \ge 2$ and start with $n=2$.
Consider
\[
[A(2)^T \lambda ]_i = [(+1 -1)^T \lambda ]_i = 
\begin{cases}
\lambda, 
& \mbox{if } i=1  \\
- \lambda,
& \mbox{if } i=2 
\end{cases}
\]
for $\lambda \in \{-1, +1\}^1$ and $i=1, 2$. The induction claim is true for $n=2$.
Assume the induction claim had been proved for $n-1$. Consider
\[
[A(n)^T \lambda ]_i =  
\begin{cases}
\sum\limits_{j=1}^{n-1} \lambda_j,
& \mbox{if } i=1 \\
- \lambda_{i-1} + [A(n-1)^T (\lambda_n, \ldots, \lambda_{s(n)})^T]_{i-1}, 
& \mbox{if } 2 \le i \le n  
\end{cases}
\]
for $\lambda = (\lambda_1, \ldots, \lambda_{s(n)})^T \in \{-1, +1\}^{s(n)}$ and 
$i=1, \ldots, n$. The first expression satisfies $-(n-1) \le \sum_{j=1}^{n-1} \lambda_j \le n-1$
for $(\lambda_1, \ldots, \lambda_{s(n)})^T \in \{-1, +1\}^{s(n)}$. Using $s(n-1) = s(n) - (n-1)$
and the induction hypothesis,
\begin{align*}
-(n-1)
& = - 1 - (n-2) \\
& \le - \lambda_{i-1} + [A(n-1)^T (\lambda_n, \ldots, \lambda_{s(n)})^T]_{i-1} \\
& \le 1 + (n-2) \\
& = n-1
\end{align*}
for $\lambda = (\lambda_1, \ldots, \lambda_{s(n)})^T \in \{-1, +1\}^{s(n)}$ and 
$i=2, \ldots, n$. This shows the induction claim.

We extend this claim to the matrix $A$, which contains the matrices $A(n)$ in the diagonal. 
Let $\lambda = (\lambda_1, \ldots, \lambda_{n \cdot s(n)})^T\in \{-1, +1\}^{n \cdot s(n)}$
and let $i \in \{1, \ldots, n^2\}$. There exists some $j \in \{1, \ldots, n\}$, such that
\[
[A^T \lambda]_i = [A(n)^T (\lambda_{(j-1)n+1}, \ldots, \lambda_{j \cdot n})^T]_i
\]
and this term satisfies the desired inequality.

The matrix $A_\pi^T$ is a permutation of the rows of $A^T$ and this completes the 
proof of the lemma.
\end{proof} 

\begin{lemma} \label{L:33}
Let $\pi$ be a permutation on $\{1,  \ldots, n^2\}$ and let $\lambda$ be a point in 
$\{-1, +1\}^{n \cdot s(n)}$, 
such that $A_\pi A^T_{\pi_1}\lambda <> \mathbf{0}$. Then $x = (x_1, \ldots, x_{n^2})^T 
= \frac{1}{2}(A_{\pi_1}^T\lambda + (n+1) \mathbf{1}_{n^2})$ satisfies 
$x \in \mathds{Z}^{n^2}$, $1 \le x_i \le n$ for $i=1, \ldots, n^2$ and 
$A_\pi x <> \mathbf{0}$.
\end{lemma}
\begin{proof} 
The point $x$ consists of integer components, since all defining variables consist of integer
components. 

Using Lemma \ref{L:32}, $A_{\pi_1}^T \lambda$ satisfies
\[
- (n-1) \le [A_{\pi_1}^T \lambda]_i \le n-1
\]
for $i=1, \ldots, n^2$. Adding $n+1$, yields
\[
2 \le [A_{\pi_1}^T \lambda + (n+1)\mathbf{1}_{n^2}]_i \le 2 n
\]
for $i=1, \ldots, n^2$ and this shows $1 \le x_i \le n$ for $i=1, \ldots, n^2$.

Using the equation $A_\pi \mathbf{1}_{n^2} = \mathbf{0}$, we obtain
\[
\begin{array}{lll}
A_\pi x 
& = &\frac{1}{2}(A_\pi A_{\pi_1}^T \lambda + (n+1)A_\pi \mathbf{1}_{n^2}) \\
& = &\frac{1}{2}A_\pi A_{\pi_1}^T \lambda \\
& <> &\mathbf{0},
\end{array}
\]
which completes the proof.
\end{proof}

Usually duality results are stated in the following sense: If there exists a primal 
feasible point and a dual feasible point and the optimal values are equal, then the primal feasible 
point solves the primal problem and the dual feasible point solves the dual problem.

Sometimes duality results are stated in another way in the literature (compare Chv\'{a}tal 
\cite[Theorem 5.1]{Chv}): If the primal problem is solvable, then the dual problem is solvable 
and the optimal values are equal. 

The relation between the primal problem and the dual problem is examined in the next theorem
and the formulation is of the second type. If the primal problem is solvable, then the dual problem 
is solvable and there exists an explicit formula for the dual solution. An analogous statement
holds for the dual problem.

\begin{theorem} \label{T:31}
The following statements hold: \\
(i) If $x$ solves $(PP)$, then $\lambda = sgn (A_{\pi_1} x)$ solves $(DP)$. \\
(ii) If $\lambda$ solves $(DP)$, then 
$x = \frac{1}{2}(A_{\pi_1}^T\lambda + (n+1) \mathbf{1}_{n^2})$ solves $(PP)$.
\end{theorem}
\begin{proof} 
(i) Assume $x = (x_1, \ldots, x_{n^2})^T$ solves $(PP)$. Then $1 \le x_i \le n$ for 
$i=1, \ldots, n^2$, $A_{\pi_r} x <> \mathbf{0}$ for $r=1, 2, 3$ and 
$A_{eq} x = g$. Let $\lambda = sgn(A_{\pi_1}x)$, then $\lambda \in \{-1, +1\}^{n \cdot s(n)}$
and $A_{\pi_r}A_{\pi_1}^T \lambda <> \mathbf{0}$ for $r=1, 2, 3$, by Lemma \ref{L:31}.
Using \cite[Theorem 5.2]{Fis},
\begin{align*}
A_{eq}A_{\pi_1}^T \lambda 
& = A_{eq}A_{\pi_1}^Tsgn( A_{\pi_1} x) \\
& = A_{eq} (A_{\pi_1}^T sgn( A_{\pi_1} x) + (n+1) \mathbf{1}_{n^2}) 
- (n+1) \mathbf{1}_{k} \\
& = 2 A_{eq}x - (n+1) \mathbf{1}_{k} \\
& = 2 g - (n+1) \mathbf{1}_{k},
\end{align*}
i.e., $\lambda$ solves $(DP)$. \\
(ii) Assume $\lambda$ solves $(DP)$. Then $\lambda \in \{-1, +1\}^{n \cdot s(n)}$, 
$A_{\pi_r}A_{\pi_1}^T \lambda <> \mathbf{0}$ for 
$r=1, 2, 3$ and $A_{eq}A_{\pi_1}^T \lambda =2 g - (n+1)\mathbf{1}_k$. Let
\[
x = \frac{1}{2}(A_{\pi_1}^T \lambda + (n+1)\mathbf{1}_{n^2}),
\]
then $x \in \mathds{Z}^{n^2}$, $1 \le x_i \le n$ for $i=1, \ldots, n^2$ and 
$A_{\pi_r} x <> \mathbf{0}$ for $r=1, 2, 3$ by Lemma \ref{L:33}. The equation
\begin{align*}
A_{eq}x 
& = \frac{1}{2}(A_{eq}A_{\pi_1}^T \lambda + (n+1)A_{eq}\mathbf{1}_{n^2}) \\
& = \frac{1}{2}(2 g - (n+1)\mathbf{1}_k + (n+1)\mathbf{1}_k) \\
& = g
\end{align*}
completes, that $x$ solves $(PP)$.
\end{proof} 

\begin{example}
We illustrate Theorem \ref{T:31} (i) with a Sudoku of size $n=4$, i.e., $s(n)=6$. The 
solution $x$ is depicted in Fig. \ref{Fig1}. At this moment it does not matter how the original problem 
had been formulated and where the givens had been located. The point 
$\lambda = sgn( A_{\pi_1}x) \in \{-1, +1\}^{n \cdot s(n)}$ consists of the values
\begin{align*}
\lambda =
& ( -1, +1, +1, +1, +1, -1, \\
& +1, -1, -1, -1, -1, -1, \\
& -1, -1, -1, -1, -1, +1, \\
& +1, +1, +1, +1, +1, +1 )
\end{align*}
and is the corresponding dual solution. This dual point describes the comparison of the values in two 
cells in the same row in Fig. \ref{Fig1}. The first component $-1$ of $\lambda$ describes, that the 
content of cell 1 in row 1 (which is a $3$) is smaller than the content of cell 2 in row 1 (which is a $4$).
\begin{figure}
\includegraphics{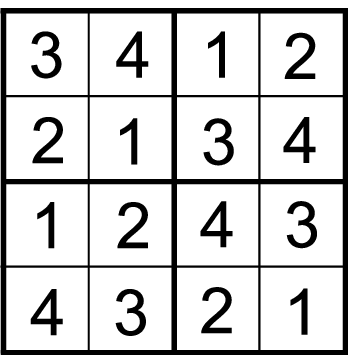}
\caption{A solution of a $4\times4$ Sudoku} 
\label{Fig1}
\end{figure}
\end{example}

This example can also be used to illustrate statement (ii) of Theorem \ref{T:31}. Defining 
$\lambda$ by the series of $+1$s and $-1$s in the example, the point 
$x = \frac{1}{2}(A_{\pi_1}^T\lambda + 5 \cdot \mathbf{1}_{4^2})$ is depicted in Fig. \ref{Fig1}.

The primal and the dual problem as introduced in Section \ref{S:PP} and \ref{S:DP} are constraint 
satisfaction problems and do not possess an objective function. Therefore it is not possible to state 
properties involving duality gaps for these problems. In the next section we replace these problems 
by two optimization problems and derive duality results for these optimization problems.

                                        %
                                        %
\section{The Primal and Dual Optimization Problems} \label{S:opt}
We introduce the primal and the dual optimization problems, which are equivalent to the
primal respectively dual problem. 

The primal optimization problem will consist of points, which have undefined components, reflecting
empty cells in a Sudoku puzzle. We describe these empty cells by the token $\infty$ and we define
$\mathds{Z}_\infty = \mathds{Z} \cup \{ \infty \}$.

When we allow points to possess infinity components, we have to extend several classic notations. 
The addition of two numbers, where one or both may be infinity, is defined as 
$\infty + x = x + \infty = \infty + \infty = \infty$ for $x \in \mathds{Z}$. We define the product 
$0 \cdot \infty = \infty \cdot 0 = 0$ and $x \cdot \infty = \infty \cdot x = \infty$ for each 
$x \in \mathds{Z}$, $x \ne 0$. Based on this extended definition of addition and multiplication, 
we extend implicitly the matrix multiplication to matrices and vectors with possible infinity components.
The token $\infty$ is different to any number, i.e., $\infty \ne x$ and $x \ne \infty$ for each 
$x \in \mathds{Z}$. In particular, $\infty$ is unequal to zero. 

This extension reflects the meaning of $\infty$ as an undefined state. Something defined and 
something undefined creates an undefined result and something undefined is different to
anything defined. The token $\infty$ has nothing to do with the commonly understanding of 
``infinitely large". It is just a placeholder for ``nothing", i.e., an empty cell in the Sudoku square.

In a classical Sudoku puzzle the term $A_{\pi_r}x <> \mathbf{0}$ with $x \in \mathds{Z}_\infty^{9^2}$
(i.e. some of the components of $x$ may be unknown) describes points where each two known values 
in the same row, the same column or the same block are distinct.

We continue with the primal feasible set
\begin{align*}
F_P = \{x=(x_1, \ldots, x_{n^2})^T \in \mathds{Z}_\infty^{n^2} \mid
& \: 1 \le x_i \le n \mbox{ or } x_i = \infty \\
& \: \mbox{for }i = 1, \ldots, n^2, \\
& \: A_{\pi_r}x <> \mathbf{0} \mbox{ for }r=1, 2, 3 \mbox{ and} \\
& \: A_{eq}x = g \}.
\end{align*}

It is easy to construct examples, where $F_P$ is empty and examples where $F_P$ is nonempty.
We define the primal objective function by $f_P(x) = \sharp \{1 \le i \le n^2 \mid x_i = \infty \}$
for each $x=(x_1, \ldots, x_{n^2})^T \in \mathds{Z}_\infty^{n^2}$. The primal objective 
function is bounded from below by $0$. The primal optimization problem is defined by
\[
(PP_{opt}) \qquad \mbox{Minimize }f_P(x) \mbox{ subject to }x \in F_P.
\]

The relevance of the primal optimization problem is the equivalence to the original primal problem
and the possible definition of solution methods for the generalized Sudoku problem.
A common strategy for solving a Sudoku puzzle creates points contained in the feasible set
of the primal optimization problem. These points consist of unpopulated cells and distinct values 
in the populated cells of each row, column and block. 

This type of solution algorithm had been proposed by Crook \cite{Cro}. His algorithm defines 
in each step a new feasible point with a lower value in the objective function until a solution 
is reached. By definition of the primal optimization problem the condition lower value of the objective 
function means the new point contains at least one more populated cell.

The primal optimal (minimal) value $min\{f_P(x)\mid x \in F_P\}$ of this optimization problem is 
denoted by $v_P$ and satisfies $v_P \ge 0$. A point $x \in F_P$ with $f_P(x) = v_P$ is called a 
solution of the primal optimization problem. 

\begin{theorem} \label{T:41}
If $F_P \ne \emptyset$, then $(PP_{opt})$ is solvable.
\end{theorem}
\begin{proof}
The primal objective function $f_P$ is bounded from below by zero and attains only integer values.
\end{proof}

The primal optimal value $v_P = 0$ if and only if one (or each) solution 
$x=(x_1, \ldots, x_{n^2})^T$ of $(PP_{opt})$ satisfies $x_i \ne \infty$ for 
$i=1, \ldots , n^2$. We describe the relation between the primal problem and the 
primal optimization problem. This relation follows in a straightforward manner from the 
definitions of the corresponding problems.

\begin{theorem} \label{T:42}
Let $x \in \mathds{Z}^{n^2}_\infty$. The following statements are equivalent: \\
(i) $x$ solves $(PP)$. \\
(ii) $x$ solves $(PP_{opt})$ and the primal optimal value $v_P = 0$.
\end{theorem}

We proceed with the dual optimization problem. The dual feasible set is denoted by
\[
F_D = \{ \lambda \in \{-1,+1\}^{n \cdot s(n)} \mid 
A_{\pi_r}A_{\pi_1}^T \lambda <> \mathbf{0} \mbox{ for }r=1, 2, 3 \}.
\]
By a special choice of $\pi_3$ it is possible to construct examples, where $F_D = \emptyset$. 
If the primal problem is a classical (solvable) Sudoku puzzle, then the dual feasible set $F_D$ is 
nonempty. The dual objective function is defined by
\[
f_D(\lambda) = \sharp \{1 \le l \le k \mid [A_{eq}A_{\pi_1}^T \lambda]_l = 2 g_{i_l} - (n+1) \} - k
\]
for each $\lambda \in \{-1,+1\}^{n \cdot s(n)}$. The dual objective function is bounded from above 
by $0$. The dual optimization problem is given by
\[
(DP_{opt}) \qquad \mbox{Maximize }f_D(\lambda) \mbox{ subject to }\lambda \in F_D.
\]
The dual optimal (maximal) value $max\{f_D(\lambda) \mid \lambda \in F_D \}$ is denoted by $v_D$
and satisfies $v_D \le 0$. A point $\lambda \in F_D$ with $f_D(\lambda) = v_D$ is called a solution of 
the dual optimization problem. 

\begin{theorem} \label{T:43}
If $F_D \ne \emptyset$, then $(DP_{opt})$ is solvable.
\end{theorem}
\begin{proof}
The dual objective function $f_D$ is bounded from above by zero and attains only integer values.
\end{proof}

It is possible to characterize dual feasible points in terms of a primal property.

\begin{theorem} \label{T:44}
Let  $x \in \mathds{Z}^{n^2}$, such that $A_{\pi_r} x <> \mathbf{0}$ for $r=1, 2, 3$ and
let $\lambda = sgn (A_{\pi_1} x)$. Then $\lambda \in F_D$. 
\end{theorem}
\begin{proof}
This follows from Lemma \ref{L:31}.
\end{proof}

\begin{theorem} \label{T:45}
Let $\lambda \in \{-1, +1\}^{n \cdot s(n)}$ and $x \in \mathds{Z}^{n^2}$, such that
$x = \frac{1}{2}(A_{\pi_1}^T\lambda + (n+1) \mathbf{1}_{n^2})$. The following 
statements are equivalent: \\ 
(i) $\lambda \in F_D$. \\
(ii) $A_{\pi_r} x <> \mathbf{0}$ for $r=1, 2, 3$.
\end{theorem}
\begin{proof}
``(i) $\Rightarrow$ (ii)" This direction follows from Lemma \ref{L:33}. \\
``(ii) $\Rightarrow$ (i)" Using the assumptions
\[
\begin{array}{lll}
A_{\pi_r}A_{\pi_1}^T \lambda 
& = & 2 A_{\pi_r}x - (n+1)A_{\pi_r}\mathbf{1}_{n^2} \\
& = & 2 A_{\pi_r}x \\
& <> & \mathbf{0}
\end{array}
\]
for $r=1, 2, 3$, i.e., $\lambda \in F_D$.
\end{proof}

The dual optimal value $v_D = 0$ if and only if one (or each) solution $\lambda$ of $(DP_{opt})$ 
satisfies $A_{eq}A_{\pi_1}^T \lambda = 2 g - (n+1)\mathbf{1}_k$. We describe the relation 
between the dual problem and the dual optimization problem. This relation follows in a straightforward 
manner from the definitions of the corresponding problems.

\begin{theorem} \label{T:46}
Let $\lambda \in \{-1,+1\}^{n \cdot s(n)}$. The following statements are equivalent: \\
(i) $\lambda$ solves $(DP)$. \\
(ii) $\lambda$ solves $(DP_{opt})$ and the dual optimal value $v_D = 0$.
\end{theorem}

                                        %
                                        %
\section{Duality Results} \label{S:dual}
In this section we collect the classical duality statements for the generalized Sudoku problem, 
namely weak duality, duality gap and strong duality. We start with the weak duality statement. 

\begin{theorem}[Weak Duality] \label{T:51}
The following statements hold: \\
(i) $f_D(\lambda) \le 0 \le f_P(x)$ for each $x \in F_P$ and $\lambda \in F_D$. \\
(ii) $v_D \le 0 \le v_P$.
\end{theorem}
\begin{proof}
We know from the definition of the primal and the dual optimization problem, that the dual 
objective function is bounded from above by zero and the primal objective function is 
bounded from below by zero. This shows (i) and implies (ii).
\end{proof}

\begin{theorem} \label{T:52}
Let $x \in F_P$, $\lambda \in F_D$ and $f_P(x) = f_D(\lambda)$. Then $x$ solves $(PP_{opt})$
with primal optimal value $v_P = 0$ and $\lambda$ solves $(DP_{opt})$ with dual optimal value
$v_D = 0$.
\end{theorem}
\begin{proof}
Using Theorem \ref{T:51}, $0 \le v_P \le f_P(x) = f_D(\lambda) \le v_D \le 0$. This implies
$f_P(x)=v_P=0$ and $f_D(\lambda)=v_D=0$.
\end{proof}

Based on the weak duality statement, we define the term duality gap by $v_P - v_D$,
which is a nonnegative value by Theorem \ref{T:51}. 
The next theorem characterizes duality gaps.

\begin{theorem}[Strong Duality] \label{T:53}
Let $F_P \ne \emptyset$ and $F_D \ne \emptyset$. The following statements are equivalent: \\
(i) $v_P = v_D$. \\
(ii) There exists a solution $x=(x_1, \ldots, x_{n^2})^T$ of $(PP_{opt})$,
such that $x_i \ne \infty$ for $i=1, \ldots, n^2$. \\
(iii) There exists a solution $\lambda$ of $(DP_{opt})$, such that
$A_{eq}A_{\pi_1}^T \lambda = 2 g - (n+1)\mathbf{1}_k$.
\end{theorem}
\begin{proof}
Assume (i) holds. By Theorems \ref{T:41} and \ref{T:43} $(PP_{opt})$ and $(DP_{opt})$ 
are solvable. Let $x$ be a solution of $(PP_{opt})$ and let $\lambda$ be a solution of 
$(DP_{opt})$. By Theorem (i) and \ref{T:52}, $f_P(x)=v_P=0$ and $f_D(\lambda)=v_D=0$,
i.e., (ii) and (iii) hold. \\
``(ii) $\Rightarrow$ (i)" Let $x=(x_1, \ldots, x_{n^2})^T$ be a solution of $(PP_{opt})$,
such that $x_i \ne \infty$ for $i=1, \ldots, n^2$. This implies $f_P(x)=0=v_P$, hence
$x$ solves $(PP)$ by Theorem \ref{T:42}. Define $\lambda = sgn(A_{\pi_1}^T x)$, 
then $\lambda$ solves $(DP)$ by Theorem \ref{T:31} (i). By Theorem \ref{T:46}, $\lambda$ 
solves $(DP_{opt})$ with $f_D(\lambda)=v_D=0$. \\
``(iii) $\Rightarrow$ (i)"  Let $\lambda$ be a solution of $(DP_{opt})$, such that
$A_{eq}A_{\pi_1}^T \lambda = 2 g - (n+1)\mathbf{1}_k$. This implies $f_D(\lambda)=v_D=0$, 
hence $\lambda$ solves $(DP)$ by Theorem \ref{T:46}. Define 
$x = \frac{1}{2}(A_{\pi_1}^T\lambda + (n+1) \mathbf{1}_{n^2})$, then $x$ solves
$(PP)$  by Theorem \ref{T:31} (ii). By Theorem \ref{T:42}, $x$ solves $(PP_{opt})$ 
with $f_P(x)=v_P=0$.
\end{proof}

It is possible to express the preceding theorem in terms of the original problems $(PP)$ and $(DP)$.
The strong duality result states, there does not exist a duality gap between $(PP_{opt})$ and
$(DP_{opt})$ if and only if the primal problem $(PP)$ is solvable respectively if and only if 
the dual problem $(DP)$ is solvable.

                                        %
                                        %

\end{document}